\documentclass[12pt]{amsart}
\usepackage{listings}
\usepackage{verbatim}
\usepackage{amsmath}
\usepackage{enumerate}
\usepackage{amssymb}
\usepackage{color}
\usepackage{url}
\usepackage{graphicx}
\usepackage{fullpage}
\usepackage{hyperref}
\usepackage{tikz}
\usepackage{CJKutf8}
\newtheorem{thm}{Theorem}[section]
\newtheorem{lem}[thm]{Lemma}
\newtheorem{mydef}[thm]{Definition}

\usepackage{listings}
\usepackage{appendix}

\usetikzlibrary{matrix}

\usepackage{xtab}
\usepackage{float}
\usepackage{color,longtable}

\newcommand\+{\;\lower\plusheight\hbox{$+$}\;}
\newcommand\lldots{\;\lower\plusheight\hbox{$\cdots$}\;}

\setbox0=\hbox{$+$}
\newdimen\plusheight
\plusheight=\ht 0 \setbox0=\hbox{$-$}
\newdimen\minusheight
\minusheight=\ht0

\setbox0=\hbox{$\cdots$}
\newdimen\cdotsheight
\cdotsheight=\plusheight

\begin{document}

\title{TRIBONACCI NUMBERS THAT ARE PRODUCTS OF TWO LUCAS NUMBERS.}
\author{
Ama Ahenfoa Quansah}

\address{
Department of Mathematics, Oklahoma State University, Stillwater, Oklahoma 74075,USA}
\email{ama.quansah@okstate.edu}

\keywords{}

\maketitle
\begin{abstract}
Let $T_{k}$ be the $k^{\textrm{th}}$ Tribonacci number and $L_{n}$ be the $n^{\textrm{th}}$ Lucas number defined by their respective recurrence relation $T_{k}=T_{k-1}+T_{k-2}+T_{k-3}$ and $L_{n}=L_{n-1}+L_{n-2}$.
In this study, we solve the Diophantine equation
$T_{k} = L_{m}L_{n}$ for positive integer unknowns $m$, $n$, and $k$ and prove our results.
\end{abstract}
\noindent
\numberwithin{equation}{section}
\allowdisplaybreaks


\section{Introduction}
Let $\{T_n\}_{n \geq 0}$ and $\{L_n\}_{n \geq 0}$ be the $n^{\textrm{th}}$ terms of the Tribonacci and Lucas numbers respectively defined via the following recurrence relation
\[
T_n = \begin{cases}
    0 & \textrm{ when } n=0\\
    1 & \textrm{ when } n=1,2\\
    T_{n-1} + T_{n-2} + T_{n-3} & \textrm{ when } n\ge 3
\end{cases}
\quad \textrm{ and } \quad
L_n = \begin{cases}
    2 & \textrm{ when } n=0\\
    1 & \textrm{ when } n=1\\
    L_{n-1} + L_{n-2} & \textrm{ when } n\ge 2.
\end{cases}
\]

The Tribonacci numbers is a third-order integer sequence satisfying  
\[
X^3 - X^2 - X - 1 = (X - \gamma)(X - \delta)(X - \bar{\delta}).
\]
This equation is called the \emph{characteristic equation}, where
\[
\gamma = \frac{1 + r_1 + r_2}{3}, \quad \delta = \frac{2 - (r_1 + r_2) + i \sqrt{3}(r_1 - r_2)}{6},
\]
with
\[
r_1 = \sqrt[3]{19 + 3\sqrt{33}} \quad \text{and} \quad r_2 = \sqrt[3]{19 - 3\sqrt{33}}.
\]
The Lucas numbers is a second-order integer sequence satisfying $x^2 - x - 1 = 0$. By considering this algebraic equation with the mentioned initial conditions, one can develop the following Binet’s formulas for all natural numbers $n$. The Binet's formulas allows us compute Tribonacci and Lucas numbers directly without needing the recursive relation. The Binet formula for the Tribonacci sequence is
\begin{equation}
\label{Binet Trib}T_n = a \gamma^{n} + b \delta^{n} + \bar{b} \bar{\delta}^{n} \quad \quad \text{for all } n \geq 0,
\end{equation}
where
\[
a = \frac{5 \gamma^{2} - 3 \gamma - 4}{22} \quad \text{and} \quad b = \frac{5 \delta^2 - 3 \delta - 4}{22},
\]
and that of the Lucas is \begin{equation}
\label{Binet Luc} L_{n}=\alpha^{n}+\beta^{n}
\end{equation}
  where $\alpha=\frac{1+\sqrt{5}}{2}$ and $\beta=\frac{1-\sqrt{5}}{2}$ are the roots of the characteristic equation of the Lucas sequence.\\
The minimal polynomial of $a$ with integer coefficients is $44X^3 - 2X - 1$, with zeros $a$, $b$, $\bar{b}$ and $\max(|a|, |b|, |\bar{b}|) < 1$.
The following numerical estimates will be useful
\begin{align}
\label{gamma estimate} 1.83 & < \gamma < 1.84; \\
\label{2} 0.73 & < |\delta| = 
\gamma^{-1/2} < 0.74;\\
0.33 & < |a| < 0.34;\\
0.25 & < |b| < 0.27.
\end{align}
For $n \geq 1$, we denote $e(n) := T_n - a \gamma^n$,so we have
\begin{equation}
\label{3}
|e(n)| < \gamma^{-\frac{n}{2}}.
\end{equation}
Furthermore,
\begin{equation}
\label{4}\gamma^{n-2} \leq T_n \leq \gamma^{n-1} \quad \quad \text{for all } n \geq 1.
\end{equation}
Let \( K \) be a number field, and denote its degree over \( \mathbb{Q} \) by \( d_K = [K : \mathbb{Q}] \). Here, $d_{\mathbb{Q}(\gamma)} = 3$ and $d_{\mathbb{Q}(\alpha)} = 2$, which implies that \( \mathbb{Q}(\gamma) \ne \mathbb{Q}(\alpha) \). Moreover, since \( \mathbb{Q}(\gamma) = \mathbb{Q}(a) \), it follows that the elements \( \gamma \), \( \alpha \), and \( a \) are real and lie in the field \( K = \mathbb{Q}(\gamma, \alpha) \), whose degree over \( \mathbb{Q} \) is \( d_K = 6 \).

We solve the Diophantine equation 
\begin{equation}
\label{Trib k} T_k = L_mL_n
\end{equation}
in positive unknowns $m$, $n$, and $k$.
We now state our main theorem which generlizes previous reuslts in the works of \cite{dacsdemir2023fibonacci, luca2023tribonacci}

\section{Statement of the Main Theorem}
\begin{thm}\label{new thm}
Let $k,m$ and $n$ be non-zero integers.
Then, \eqref{Trib k} is satisfied for the triples 
\begin{equation}
(k,m,n)\in \{(1,1,1), (2,1,1), (4,1,3), (5,1,4), (8,3,5)\}
\end{equation}
\end{thm}

\section{Preliminary}

\begin{mydef}
\label{logarithmic height}
Let $\eta$ be an algebraic number of degree $d$ with the minimal polynomial
\[
f(X) := \sum_{j=0}^{d} a_j X^{d-j} = a_0 X^d + a_1 X^{d-1} + \cdots + a_d = a_0 \prod_{i=1}^{d} (X - \eta^{(i)}) \in \mathbb{Z}[X],
\]
where $a_0 > 0$, $a_i$'s are relatively prime integers, and $\eta^{(i)}$ is the $i^{\textrm{th}}$ conjugate of $\eta$.

The \textbf{logarithmic height} of $\eta$, denoted by $h(\eta)$, is defined by
\[
h(\eta) = \frac{1}{d} \left( \log |a_0| + \sum_{i=1}^{d} \log \left(\max \left\{ |\eta^{(i)}|, 1 \right\}\right) \right).
\]
\end{mydef}

We now record some basic properties of the logarithmic height.
For algebraic numbers $\eta, \gamma$, and any integer $s$, we have
\begin{align*}
h(\eta\pm \gamma) &\leq h(\eta) + h(\gamma) + \log 2\\
h(\eta\gamma^{\pm 1})&\leq h(\eta)+h(\gamma),\\
h(\eta^s) &= |s| h(\eta) \quad (s \in \mathbb{Z}).
\end{align*}

In what follows we often write $\mathbb{F}$ to denote a real number field of degree $D$ over $\mathbb{Q}$.
This means $\mathbb{F}$ is a finite extension of $\mathbb{Q}$ consisting entirely of real numbers, where every element of $\mathbb{F}$ can be expressed as a linear combination of $D$ basis elements with coefficients in $\mathbb{Q}$.

\begin{lem}
\label{Lemma 2}
Let $n$ be a positive integer.
Then,
\begin{align*}
\alpha^{n-2} & \leq F_n \leq \alpha^{n-1}\\
\alpha^{n-1} & \leq L_n \leq 2\alpha^n \\
|\beta|^{-(n-2)} = |\beta^{-(n-2)}| & \leq F_n \leq |\beta^{(n-1)}| = |\beta|^{-(n-1)}\\
|\beta|^{-(n-1)} = |\beta^{-(n-1)}| & \leq L_n \leq |\beta^{-(n+1)}| = |\beta|^{-(n+1)}
\end{align*}

\end{lem}

\begin{proof} 
The above lemma is taken from \cite[Lemma 2]{dacsdemir2023fibonacci}. We provide a proof of the first and last inequalities by induction.

We first prove
\[\alpha^{n-2}\leq F_{n}\leq\alpha^{n-1}.\]
where $\alpha=\frac{1+\sqrt{5}}{2}\approx 1.618$. 
We verify the inequality for the base cases $n=1,2$
\begin{align*}
\text{For }n=1, \quad & \alpha^{-1}\approx 0.618 \leq F_{1}=1\leq \alpha^{0}=1\\
\text{For }n=2, \quad & \alpha^{0}=1\leq F_{2}=1\leq \alpha^{1}\approx 1.618.
\end{align*}
Next we assume $n\geq 3$
\begin{equation}
\label{inductive 1} \alpha^{n-2}\leq F_{n}\leq \alpha^{n-1} \quad \textrm{and} \quad \alpha^{n-3}\leq F_{n-1}\leq \alpha^{n-2}.
\end{equation}
We need to show that
\[
\alpha^{(n+1)-2}\leq F_{n+1}\leq \alpha^{(n+1)-1}=\alpha^{n-1}\leq F_{n+1}\leq \alpha^{n}.
\]
We have that $F_{n+1}=F_{n}+F_{n-1}$ and from the inequalities \eqref{inductive 1}
\[F_{n}\geq \alpha^{n-2}\quad \textrm{and} \quad F_{n-1}\geq \alpha^{n-3}.
\]
Therefore,
\begin{align*}
F_{n+1}=F_{n}+F_{n-1}&\geq\alpha^{n-2}+\alpha^{n-3}\\
F_{n+1}&\geq \alpha^{n-3}(\alpha +1)
\end{align*}
Note that
$\alpha +1=\frac{1+\sqrt{5}}{2}+1=\frac{3+\sqrt{5}}{2}$  and $\alpha^{2}=\frac{3+\sqrt{5}}{2}$.
Thus 
\begin{equation}
\label{alpha + 1}
\alpha +1=\alpha^{2}
\end{equation}
Using this we can conclude that
\[
F_{n+1}\geq \alpha^{n-3}\cdot\alpha^{2}=\alpha^{n-1}.
\]
From the inequalities \eqref{inductive 1} we know that
\[
F_{n}\leq \alpha^{n-1},\quad F_{n-1}\leq \alpha^{n-2}.
\]
Since, $F_{n+1}=F_{n}+F_{n-1}\leq \alpha^{n-1}+\alpha^{n-2}$ we obtain that
\[
F_{n+1}\leq \alpha^{n-2}(\alpha +1).
\]
From \eqref{alpha + 1} we know that $F_{n+1}\leq \alpha^{n-2}\cdot\alpha^{2}\leq\alpha^{n}$.
So, we conclude that
\[
\alpha^{n-1}\leq F_{n+1}\leq \alpha^{n}.
\]
Thus, if the inequality holds for $n$, it holds for $n+1$.

Next we prove the last inequality.
Recall that $\beta=\frac{1-\sqrt{5}}{2}\approx 0.618$ and observe that $|\beta|^{-1}=\alpha$, so the inequality can becomes 
\[
\alpha^{n-1}\leq L_{n}\leq \alpha^{n+1}.
\] 
We verify the base cases $n=1,2$
\begin{align*}
\textrm{For } n=1, & \quad 1\leq L_{1}=1\leq \alpha^{2}\approx 2.618\\
\textrm{For } n=2, & \quad \alpha \approx  1.618\leq L_{2}=3\leq\alpha^{3} \approx 4.2361.
\end{align*}
We now show the inequality is satisfied for $n\geq 2$:
\begin{equation}
\label{inductive 2}
\alpha^{n-1}\leq L_{n}\leq \alpha^{n+1} \quad \textrm{and} \quad \alpha^{n-2}\leq L_{n-1}\leq \alpha^{n}
\end{equation}
We now obtain bounds for the $n+1^{\textrm{th}}$  Lucas number in terms of $\alpha$.
\begin{align*}
L_{n+1}  =L_{n}+L_{n-1} & \geq \alpha^{n-1}+\alpha^{n-2} \quad \textrm{using the inequalities in \eqref{inductive 2}} \\ 
& \geq \alpha^ {n-2}(\alpha+1) \\
& \geq\alpha^{n-2}\cdot \alpha^{2}=\alpha^{n} 
\quad \textrm{using \eqref{alpha + 1}}.
\end{align*}
Similarly,
\begin{align*}
    L_{n+1}  =L_{n}+L_{n-1} & \leq \alpha^{n+1}+\alpha^{n} \quad \textrm{using the inequalities in \eqref{inductive 2}} \\ 
    & \leq \alpha^{n}(\alpha+1) \\
    & \leq\alpha^{n}\cdot \alpha^{2}=\alpha^{n+2} 
    \quad \textrm{using \eqref{alpha + 1}}. \qedhere
\end{align*}

\end{proof}

\begin{mydef}
\label{Lambda def}
Let $\eta_1, \eta_2, \ldots, \eta_s$ be positive algebraic numbers in the real number field $\mathbb{F}$ of degree $D$ and let $b_1, b_2, \ldots, b_s$ be non-zero rational numbers.
Define
\[
\Lambda := \eta_1^{b_1} \eta_2^{b_2} \cdots \eta_s^{b_s} - 1 \quad \text{and} \quad B := \max \{|b_1|, |b_2|, \ldots, |b_s|\}.
\]
Further define $A_1, A_2, \ldots, A_s$ be the positive real numbers such that
\[
A_j \geq \max \left\{ Dh(\eta_j), \log |\eta_j|, 0.16 \right\} \quad \text{for all} \; j = 1, 2, \ldots, s.
\]
\end{mydef}

With notation as above, we state Matveev's theorem \cite{dacsdemir2023fibonacci} and a lemma by Dujella--Peth\"{o} \cite{dacsdemir2023fibonacci} which is used crucially in the remainder of this work.

\begin{thm}[Matveev] \cite[Theorem 1]{dacsdemir2023fibonacci}
\label{Matveev}
With notation introduced above, the following inequality holds for any non-zero algeraic number $\Lambda$ in the real field $\mathbb{F}$
\[
\log |\Lambda| > -1.4 \times 30^{s+3} \times s^{4.5} \times D^2 \times (1 + \log D) \times (1 + \log B) \times A_1 \times A_2 \times \cdots \times A_s.
\]
\end{thm}

\begin{lem}[{Dujella--Peth\"{o}}] \cite[Lemma 1]{dacsdemir2023fibonacci}
\label{Dujella Petho}
Let $M$ be a positive integer, $\frac{p}{q}$ be a convergent of the continued fraction of the irrational $\tau$ such that $q > 6M$, and let $A$, $B$, $\mu$ be positive rational numbers with $A > 0$ and $B > 1$. Let $\epsilon = \Vert \mu q\Vert - M\Vert \tau q\Vert$, where $\Vert \cdot \Vert$ is the distance from the nearest integer. If $\epsilon > 0$, then there is no integer solution $(m, n, k)$ of inequality
\[
0 < m\tau - n + \mu < AB^{-k} \textrm{ where }
m \leq M \text{ and }  k \geq \frac{\log (Aq/\epsilon)}{\log B}.
\]
\end{lem}

\begin{lem}[\cite{sanchez2014linear}]
\label{Lemma l}
If $l\geq1$, $H>(4l^{2})^{l}$, and $H>L/(\log L )^{l}$, then 
\[
L<2^{l}H(\log H)^{l}.
\]
\end{lem}

The following result is taken from \cite[Lemma 3]{dacsdemir2023fibonacci}
but the proof is added for convenience of the reader.

\begin{lem}[Ddamulira--Luca--Rakotomalala]
\label{Lemma 3}
Consider the real number $z$ in the interval $\left( -\frac{1}{2}, \frac{1}{2} \right)$.
Then
\[
|z| < 2 |e^z - 1|.
\]
\end{lem}

\begin{proof}
We consider the cases $x<0$ and $x>0$.\\
\textbf{Case I:} $x>0$\\
For all positive real numbers $x$, it follows from the truncated Taylor series expansion of the exponential function $e^x$ that $e^{x} \geq 1+x$.
Thus,
\begin{align*}
|e^{x}-1|&\geq |x|\\
|x|&\leq |e^{x}-1|\\
2|x|=2&\leq 2|e^{x}-1|\\
x<2x&\leq 2|e^{x}-1|
\end{align*}
Therefore, $|x|<2|e^{x}-1|$.

\textbf{Case II:} $x<0$\\
For negative $x$, the exponential function satisfies $e^{x}\leq 1+x$.
Thus, $e^{x}-1 \leq x$.
Since $e^{x}-1\leq x<0$ we have
\[
|e^{x}-1|=-(e^{x}-1) \quad \textrm{and} \quad |x|=-x
\] 
Furthermore, since $\quad-2(e^{x}-1)\geq-2x$  and $-2x>-x$, we conclude
\[
|x|=-x<-2x\leq 2|e^{x}-1| \implies |x|<2|e^{x}-1|. \qedhere
\]   
\end{proof}

\section{Proof of Main Results}
Let $(k, m, n)$ be a solution of the Diophantine equation \eqref{Trib k}.
Let's suppose $1 \leq m \leq n$ and $k \geq 1$ because $L_1 = T_1 = T_{2}=1$. We assume that $m \geq 2$, because for $m = 1$, we get $T_k = L_n$.\\

Here, we find upper and lower bounds for $k$ in terms of $m$ and $n$. From \eqref{4} and the last equation in Lemma \ref{Lemma 2}, we can see that
\begin{equation}
\label{Tk inequalities}\gamma^{k-2}<T_{k}=L_{m}L_{n}<|\beta|^{-(n+m+2)}\quad \textrm{and} \quad |\beta|^{-(m+n-2)}<\gamma^{k-1}
\end{equation}
Taking logs of the first inequality in \eqref{Tk inequalities} we have,
\begin{align*}
\label{up and low k} (k-2)\log \gamma&<-(n+m+2)\log |\beta|\\
\frac{(k-2)\log \gamma}{\log \gamma}&<\frac{-(n+m+2)\log |\beta|}{\log \gamma}\\
k-2&<\frac{-(n+m+2)\log |\beta|}{\log \gamma}\\
k&<-(n+m+2)\frac{\log |\beta|}{\log \gamma}+2
\end{align*}
From the estimation of $\gamma$ in \eqref{gamma estimate}, we have 
\begin{equation}
\label{upper bound for k}
k<\frac{\log |\beta|}{\log \gamma}(-m-n)+3.8
\end{equation}
Taking logs of the second inequality in \eqref{Tk inequalities} gives,
\begin{align*}
-(m+n-2)\log |\beta|&<(k-1)\log \gamma\\
\frac{-(m+n-2)\log |\beta|}{\log \gamma}&<\frac{(k-1)\log \gamma}{\log \gamma}\\
-(m+n-2)\frac{\log |\beta|}{\log \gamma}&<k-1\\
-(m+n-2)\frac{\log |\beta|}{\log \gamma}+1&<k
\end{align*}

\begin{equation}
\label{lower bound for k}
\frac{\log|\beta|}{\log \gamma}(-m-n)-0.2<k
\end{equation}

Combining the inequalities in \eqref{upper bound for k} and \eqref{lower bound for k}, we find these upper and lower bounds for k
\[\frac{\log |\beta|}{\log \gamma}(-m-n)-0.2<k<\frac{\log |\beta|}{\log \gamma}(-m-n)+3.8\]

Define a non-zero algebraic number $\Lambda_1$ in terms of $a$, $\gamma$, $\beta$, $m$, $n$, and $k$. 
We then use this algebraic number while applying Matveev's Theorem.

Recall the Binet's formulae for the Tribonacci numbers \eqref{Binet Trib} and Lucas numbers \eqref{Binet Luc}
\[
T_{n}=a \gamma^{n} + b \delta^{n} + \bar{b} \bar{\delta}^{n} \hspace{0.5cm}  \text{and} \hspace{0.5cm} L_{n}=\alpha^{n}+\beta^{n}
\]
Suppose the $k^{\textrm{th}}$ Tribonacci number $T_k$ can be written as a product of two Lucas numbers.
Then we can perform the following manipulations

\begin{align*}
T_{k}&=L_{m}L_{n}\\
a \gamma^{k} + b \delta^{k} + \bar{k} \bar{\delta}^{k}&=(\alpha^{m}+\beta^{m})(\alpha^{n}+\beta^{n})\\
\frac{a \gamma^{k} + b \delta^{k} + \bar{k} \bar{\delta}^{k}}{\beta^{m+n}}&=\frac{\alpha^{m+n}+\alpha^{m}\beta^{n}+\alpha^{n}\beta^{m}+\beta^{m+n}}{\beta^{m+n}}\\
a \gamma^{k}\beta^{-(m+n)}+(b \delta^{k} + \bar{k} \bar{\delta}^{k})\beta^{-(m+n)}&= \alpha^{m+n}\beta^{-(m+n)}+\alpha^{m}\beta^{n}\beta^{-(m+n)}+\alpha^{n}\beta^{m}\beta^{-(m+n)}+1\\
|a\gamma^{k} \beta^{-(m+n)}-1|&=\big|-(b\delta^{k}+\bar{b}\bar{\delta^{k}})\beta^{-(m+n)} +\alpha^{m+n}\beta^{-(m+n)}+\alpha^{m}\beta^{n}\beta^{-(m+n)}\big|
\end{align*}
In view of Definition \ref{Lambda def}, we define 
\[
\Lambda_{1}:=|a\gamma^{k} \beta^{-(m+n)}-1|\neq 0
\]
where $\eta_{1}=\gamma$, $\eta_{2}=|\beta|$, $\eta_{3}=a$ and $b_{1}=k$, $b_{2}=-m-n$, $b_{3}=1$.
And, we have 
\begin{equation}
\label{Lambda1 in terms of beta and m}
| \Lambda_1 |< \frac{5.74}{|\beta|^{2m}}.
\end{equation}

Here, we perform algebraic manipulations and set the stage for applying Matveev's theorem.
First, we compute the logarithmic heights of the algebraic numbers $\eta_i$ corresponding to our algebraic number $\Lambda_1$.
Recall the definition of the logarithmic height of an algebraic number from Definition~\ref{logarithmic height}.
Since the minimal polynomials for $\gamma$ and $\beta$  are $x^{3}-x^{2}-x-1$ and $x^{2}-x-1$ respectively, we compute $h(\eta_{1}=\gamma)$, $h(\eta_{2}=|\beta|)$,  and $h(\eta_{3} = a)$ using the above formula. Recall that $d_{\mathbb{K}} = 6$. 
\begin{align*}
h(\eta_{1}) &=\frac{1}{3}\bigg(\log|1|+\log(\max\{|\eta^{(1)}|,1\})+\log(\max\{|\eta^{(2)}|,1\})\bigg)\\
     &=\frac{1}{3}\bigg(0+\log(\max\{\gamma,1\})+\log(\max\{|\beta|,1\})\bigg)\\
     &=\frac{1}{3}\log\gamma<0.204 \quad (\textrm{from the numerical estimate of gamma in \ref{gamma estimate}})
\end{align*}
Similarly,
\[
h(\eta_{2})=\frac{1}{2}\bigg(\log|1|++\log(\max\{|\gamma|,1\})\log(\max\{|\beta|,1\})\bigg)
=\frac{1}{2}\log|\gamma|< 0.303  
\]
For $\eta_{3}=\sqrt{5}$, we fix the minimal polynomial $x-\sqrt{5}$ and use the definition to obtain
\[
h(\eta_{3})=\log|1|+\log(\max\{a,1\}) =\log a <2.172
\]

Then for each algebraic number $\eta_1 = \gamma$, $\eta_2 = |\beta|$, and $\eta_3 = a$ we define positive real numbers $A_{1}$ , $A_{2}$, and $A_{3}$ such that
\[
A_{j}\geq \max \{6h(\eta_{j}), \ |\log \eta_{j}|, \ 0.16\} \quad \textrm{ for } 1\leq j \leq 3
\] 
where we use the fact that $D=6$. More explicitly, $A_{j}=Dh(\eta_{j})$ which get

   \[ A_1=1.23,\quad A_2 =1.82 ,\quad A_3 =13.03\] 
   
From the definition of $B$ in Definition~\ref{Lambda def}, we have
\[
B :=\max \{|b_{1}|, \ |b_{2}|, \ |b_{3}|\} = \max \{k, \ n+m, \ 1 \}=m+n
\] 

We then apply Matveev's Theorem to get an upper bound for $m$ in terms of $n$

\begin{lem}
\label{Lemma 3.3.2a}
The following inequality holds for an algebraic number $Lambda_{1}$,
\[-7.28\times 10^{14}\times \log (m+n)<\log (| \Lambda_{1} | )< \log 5.62- 2m\log |\beta|.
\]
\end{lem}

\begin{proof}
Here, $\Lambda_{1}=a\gamma^{k}|\beta|^{-(m+n)}-1$, $s=3$, and $D=6$.
Further recall that we showed $B= m+n$.
Then applying Matveev's theorem we have,
\begin{align*}
\log(|\Lambda_1|) & > -1.4\times 30^{s+3}\times s^{4.5}\times D^{2}\times (1+\log D)\times (1+\log B)\times A_{1}\times A_{2}\times A_{3}\\
&> -1.4\times 30^{3+3}\times 3^{4.5}\times 6^{2}\times (1+\log 6)\times (1+\log m+n)\times 1.23 \times \log 1.82 \times 13.03\\
& >-1.4\times 7.29\times 10^{8}\times 140.296 \times 6\times 1.23 \times 1.82 \times 13.03\times (1+\log 4n)\\
&> -7.28\times 10^{14}\times \log (m+n).
\end{align*}

Recall that we had obtained an upper bound for $\Lambda_{1}$ as
\[
| \Lambda_{1} | =\Bigg|a\gamma^{k}|\beta|^{-(n+m})-1\bigg| <\frac{5.62}{|\beta|^{2m}}.
\]
Taking log of both sides
\begin{align*}
    \log | \Lambda_{1} | &<\log\bigg(\frac{5.62}{|\beta|^{2m}}\bigg)\\
    &<\log 5.62-\log |\beta|^{2m}\\
    &< \log 5.62- 2m\log |\beta|.
\end{align*}
\end{proof}

Using the bound for $\Lambda_{1}$ to get a bound for $m$,the inequality from Lemma~\ref{Lemma 3.3.2a} yields
\begin{align*}
-7.28\times 10^{14}\times \log (m+n)&< \log 5.62- 2m\log |\beta|\\
2m\log |\beta| &< \log 5.62+ 7.28\times 10^{14}(1+\log m+n)\\
m \log |\beta|&<\frac{1}{2}(\log 5.62+7.28\times 10^{14}(1+\log m+n))\\
m \log |\beta| &< 3.85\times 10^{14}(\log m+n).
\end{align*}

The following step is dedicated to attaining an absolute upper bound for $n$.
As we had done previously, we use the Binet's formula to obtain 
\begin{align*}
T_{k}&=L_{m}L_{n}\\
a\gamma^{k}+b\delta^{k}+\bar{b}\bar{\delta^{k}}& =L_{m}(\alpha^{n}+\beta^{n})\\
\frac{a\gamma^{k}}{L_{m}\beta^{n}} +\frac{b\delta^{k}}{L_{m}\beta^{n}}+\frac{\bar{b}\bar{\delta^{k}}}{L_{m}\beta^{n}}& =\frac{\alpha^{n}}{\beta^{n}}+1\\
\left|\frac{a\gamma^{k}}{L_{m}|\beta|^{n}}-1\right|&<\bigg(\frac{1}{3}+\frac{1}{|\beta|^{2}}\bigg)|\beta|^{-n}
\end{align*}
We define
\begin{equation}
\label{Lambda2 in terms of gamma and m}
\Lambda_{2} := \bigg| \bigg(\frac{a}{L_{m}}\bigg)\gamma^{k}|\beta|^{-n}-1 \bigg|
\end{equation}
and repeating the same calculations as before, we can conclude that 
\[0<\Lambda_{2} < 1.5|\beta|^{-n}\textrm{ and } 2n<11.6\times 10^{29}(\log(2n))^2.\]
 
 Applying Lemma \ref{Lemma l} with with $t=2,L=2n$ and $H=11.6\times 10^{30}$ yields 
\[
n <2.24 \times 10^{34}.
\]

The bounds we have for our indices $k,m,n$ are too huge so next we try to get better bounds using the Dujella--Peth\"{o} Lemma.

\noindent 
Define $\Gamma_1$ such that 
\[
\Lambda_{1} = |a\gamma^{k}|\beta|^{-(m+n)}-1| = |\exp(\Gamma_{1})-1| < \frac{5.62}{|\beta|^{2m}}.
\]

Specifically \[
\Gamma_{1}:= k\log \gamma - (m+n)\log|\beta|+\log a
\]

From Lemma \ref{Lemma 3}, we have that
\begin{equation}
\label{Gamma 1 bound}
0< \left|\frac{\Gamma_1}{\log |\beta|}\right| = \left|\frac{k\log \gamma}{\log|\beta|}-(m+n)+\frac{\log (a)}{\log |\beta|}\right|<23.37|\beta|^{-2m}.
\end{equation}

Applying the Dujella-Peth\"{o} lemma \ref{Dujella Petho} to the last inequality by considering $M=4.2\times 10^{34}$ ($M>k$) and $\tau=\frac{\log\gamma}{\log|\beta|}$, computing the continued fraction expansions of $\tau$ yields
\[
\frac{p_{71}}{q_{71}} = \frac{452544523220541439982411039079661113}{357364106913532334879636629737733870}
\]
This means that $6M<q_{71}=357364106913532334879636629737733870$. As a result, we obtain 
\[
\epsilon= \Vert \mu q_{47}\Vert -M\Vert\tau q_{47}\vert, \quad \epsilon > 0.325 \quad \mu_{m}=\frac{\log a}{\log|\beta|}
\]
From the inequality \eqref{Gamma 1 bound} and the Dujella-Peth\"{o} Lemma (see Lemma~\ref{Dujella Petho}), we see that $A=23.37$, $B=\alpha^2$, and $k=m$ so we conclude that $m\leq 361$.

Next consider $\Lambda_{2}$ and from \eqref{Lambda2 in terms of gamma and m}, define $\Gamma_2$ such that 
\[
\Lambda_{2} = |a/L_{m}\gamma^{k}|\beta|^{-n}-1| = |\exp(\Gamma_{2})-1| < \frac{1.5}{|\beta|^{-n}}.
\]
then we have 
\[
0 < \left|k \log \gamma - n \log |\beta| + \log \left( \frac{a}{L_m} \right)\right| < 3 |\beta|^{-n}.
\]

This implies that
\begin{equation}
\label{bound Gamma 2}
0 < \left|k\frac{\log\gamma}{\log|\beta|}-n + \frac{\log(a/L_{m})}{\log |\beta|} \right| < 8 |\beta|^{-n}.
\end{equation}

We then apply Lemma~\ref{Dujella Petho} with $k=n$
\[
\tau = \frac{\log\gamma}{\log |\beta|},\quad \mu=\frac{\log(a/L_{m})}{\log|\beta|}, \quad A = 8, \quad B = |\beta|, \quad M = 4.72 \cdot 10^{34}.
\]

With the help of Mathematica, we find that the 41st convergent of $\tau$ is
\[
\frac{p_{41}}{q_{41}} = \frac{237161759629456603958}{187281242121494666147}.
\]
It satisfies $q_{41} > 6M$ and $\epsilon > 0.273 > 0$. Hence, inequality \eqref{bound Gamma 2} has no solution for
\[
n \geq \frac{\log(8q_{41}/\epsilon)}{\log |\beta|} > 375.1
\]

Thus, we obtain $n \leq 375$.
We now check for the solutions of equation \eqref{Trib k} for $ m \leq 361$ and $n \leq 375$. This was done  with a small program in Python. 
This finishes the proof. 

\vspace{10cm}
\begin{appendices}
\section{Python Code for Verifying Solutions}

\begin{verbatim}
# Function to generate Lucas numbers up to n
def generate_lucas(n):
    lucas_numbers = [2, 1]  # L_0 = 2, L_1 = 1
    for i in range(2, n + 1):
        lucas_numbers.append(lucas_numbers[-1] + lucas_numbers[-2])
    return lucas_numbers

# Function to generate Tribonacci numbers up to n
def generate_tribonacci(n):
    trib_numbers = [0, 0, 1]  # T_0 = 0, T_1 = 0, T_2 = 1
    for i in range(3, n + 1):
        trib_numbers.append(trib_numbers[-1] + trib_numbers[-2] + trib_numbers[-3])
    return trib_numbers

# Function checks if a Tribonacci number is a product of two Lucas numbers
def find_lucas_product_for_tribonacci(trib_num, lucas_numbers):
    for i in range(len(lucas_numbers)):
        for j in range(i, len(lucas_numbers)):
            if trib_num == lucas_numbers[i] * lucas_numbers[j]:
                return (trib_num, lucas_numbers[i], lucas_numbers[j])
    return (trib_num, None, None)

# Generate Lucas and Tribonacci numbers
lucas_numbers = generate_lucas(500)  # Generate Lucas numbers up to L_50
trib_numbers = generate_tribonacci(500)  # Generate Tribonacci numbers up to T_50

# Checks if each Tribonacci number can be written as a product of two Lucas numbers
results = []
for trib_num in trib_numbers:
    result = find_lucas_product_for_tribonacci(trib_num, lucas_numbers)
    results.append(result)

# Display the results
print(f"{'Tribonacci Number':<20} {'Lucas Factor 1':<20} {'Lucas Factor 2'}")
for result in results:
    trib_num, factor1, factor2 = result
    if factor1 is not None:
        print(f"{trib_num:<20} {factor1:<20} {factor2}")
    else:
        print(f"{trib_num:<20} {'None':<20} {'None'}")

\end{verbatim}

\end{appendices}

 \bibliographystyle{elsarticle-num}
\bibliography{References}




\end{document}